\documentclass[11pt]{amsart}
\usepackage{amsmath,amssymb}
\usepackage[cp1251]{inputenc}
\usepackage{graphicx}

\newtheorem{thm}{Theorem}[section]
\newtheorem{prop}[thm]{Proposition}
\newtheorem{conj}[thm]{Conjecture}
\newtheorem{defi}[thm]{Definition}
\newtheorem{lem}[thm]{Lemma}

\theoremstyle{definition}

\newcommand{\x}{\mathbb{X}}
\newcommand{\V}{{\mathcal V}}                    %
\newcommand{\op}[1]{\operatorname{#1}}             
\newcommand{\mps}[3]{#1:#2\rightarrow #3}           
\newcommand{\matop}[3]{\mathop{#1}\limits_{#2}^{#3}} 

\newcommand{\pam}[4]{#1\hookleftarrow #2\mathop{\rightarrow}
\limits^{#3}#4}                                 
\newcommand{\cal}[1]{{\mathcal #1}}                

\newcommand{\ii}{^{-1}}                          
\newcommand{\mean}{\rightleftharpoons}          

\newcommand{\ANE}{\operatorname{ANE}}          
\newcommand{\cov}{\operatorname{cov}}           
\newcommand{\Id}{\operatorname{Id}}             %
\begin{document}
\title[On Murayama's theorem on extensor properties ...]
{On Murayama's theorem on extensor properties of $G$-spaces of
given orbit types}

\author[Sergei Ageev]{Sergei Ageev}
 \address[S. Ageev]{Department  of Mathematics and Mechanics,
 Belarus State University, Minsk, Independence Avenue, 4, Belarus 200050}
 \email{ageev\_serhei@yahoo.com}

\author[Du\v{s}an Repov\v{s} ]{Du\v{s}an Repov\v{s}}
 \address[D. Repov\v{s}]{Faculty of Mathematics and Physics, and Faculty of
Education, University of Ljubljana, P.O.B. 2964, Ljubljana,
Slovenia 1001}
 \email{dusan.repovs@guest.arnes.si}



\date{\today}
\maketitle

\begin{abstract} We develop a method of extending actions of
compact transformation groups which is then applied to the problem
of preservation of equivariant extensor property by passing to a
subspace of given orbit types.
\end{abstract}

\medskip

 \section{Introduction}
 The problem of topological characterization of simplicial
complexes motivated Borsuk to introduce an  important class of
spaces, namely absolute neighborhood retracts ($\op{ANR}$-spaces)
which turned out to be a wider class than simplicial complexes,
but intimately close to them with respect to other properties.
 As it was shown by Dugundji \cite{Du}, each $\op{ANR}$-space is
characterized by the property that it admits an arbitrarily fine
domination by simplicial complexes.

In equivariant topology the role of simplicial complexes is played
by  $G$-$\op{CW}$-complexes, and the role  of absolute
neighborhood retracts is played by $G$-$\ANE$-spaces. In the
equivariant case Dugundji's characterization mentioned above
consists of
the following plausible statement:
\begin{conj}\label{one141} Let $G$ be a compact group. Then
any metric $G$-space $\Bbb X\in G$-$\ANE$ {\it admits an
arbitrarily fine domination by $G$-$\op{CW}$-complexes}, i.e. for
each cover $\omega\in\cov\Bbb X$ there exist a
$G$-$\op{CW}$-complex $\Bbb Y$ and $G$-maps $\Bbb
X\matop{\rightarrow}{}{f}\Bbb Y\matop{\rightarrow}{}{g}\Bbb X$
such that $g\circ f$ and $\op{Id}_{\Bbb X}$ can be joined by an
$\omega$-$G$-homotopy.
  \end{conj}
 This has so far been settled only for two special cases: for compact metric
$G$-$\ANE$-spaces \cite{Ag1}, and for metric $G$-$\ANE$-spaces
with an action of a zero-dimensional compact group \cite{Ag2}. In
general this is still a conjecture.

 We consider one more question concerning the closeness of
$G$-$\op{CW}$-complexes and $G$-$\ANE$-spaces. It is well-known
 that if $\Bbb Y$ is a $G$-$\op{CW}$-complex then for
each closed family ${\cal C}\subset\op{Orb}_G$ of orbit types, the
$G$-subspace  $\Bbb Y_{\cal C}$ of points  of orbit type ${\cal
C}$ is also a $G$-$\op{CW}$-complex  \cite{DT}. Murayama \cite{Mu}
proved that for the family of orbit type ${\cal C}=\{(K)\mid
(K)\ge(H)\}$ the complete analogy with $G$-$\op{CW}$-complexes is
preserved.
    \begin{thm}\label{one14}
 If $G$ is a compact Abelian Lie group and
$\Bbb X$ is a metric $G$-$\ANE$-space then  $\Bbb X^{(H)}=G\cdot
\Bbb X^H\in G$-$\ANE$ for each closed subgroup  $H<G$.
  \end{thm}
 If a compact Lie group  $G$ is nonabelian, then $\Bbb X^H$ admits the action of
the normalizer $\op{N}(H)$ of $H$ and cannot in general be
endowed with the action of $G$. Therefore there exists a point
$x\in \Bbb X^G\subset\Bbb X^H$ without $G_x$-slices in $\Bbb X^H$.
Since exactly this  argument was the key in the proof of
\cite[Proposition 8.7]{Mu}, the case of such a group cannot be
considered to be settled.

\smallskip We show that Conjecture~\ref{one141} implies
the validity of Theorem~\ref{one14} for arbitrary compact groups.
Let  $\omega\in\cov\Bbb X$, and let $\Bbb Y$ be a
$G$-$\op{CW}$-complex and $\Bbb X\matop{\rightarrow}{}{f}\Bbb
Y\matop{\rightarrow}{}{g}\Bbb X$ $G$-maps such that $g\circ
f\matop{\simeq_G}{}{}\op{Id}_{\Bbb X}\op{rel}[\omega]$. Since
$\Bbb Y^{(H)}$ is a $G$-$\op{CW}$-complex and therefore $\Bbb
Y^{(H)}\in G$-$\ANE$ (\cite[ Theorem 12.5]{Mu}), $\Bbb X^{(H)}$
admits an arbitrarily fine domination by $G$-$\ANE$-complexes. By
\cite[Theorem 9.2]{Mu} it follows that $\Bbb X^{(H)}\in
G$-$\ANE$.

Based on the remark made above it thus
follows that Theorem~\ref{one14} is proved for compact metric $G$-$\ANE$-spaces and for
metric $G$-$\ANE$-spaces with an action of a zero-dimensional
compact group. In the present paper we develop a new approach
based on a reduction of the problem to that of extending the
action of groups which is also of independent interest. As a
result we obtain the following
    \begin{thm}\label{one1411}
 Let  $G$ be a compact Lie group, $\Bbb X$  a $G$-$\ANE$-space and
${\cal C}\subset\op{Orb}G$ a saturated family of orbit types. Then
$\Bbb X_{\cal C}\subset\Bbb X$ is $G$-$\ANE$.
   \end{thm}
 Since for each $H<G$ the family of orbit types $\{(K)\mid
(K)\ge(H)\}$  evidently  satisfies the hypotheses of this theorem,
the strengthening of Murayama's theorem is valid for arbitrary
action of compact Lie groups. We conjecture that the further
generalization of Theorem~\ref{one1411} for compact group action
on a metric space is also valid, provided that ${\cal
C}\subset\op{Orb}G$ is a saturated family in which the
intersection ${\cal C}\cap{\cal E}$ with the family ${\cal E}$ of
extensor orbit types is cofinal in ${\cal E}$ \footnote{Recall
that the orbit type $(H)$ is called {\it extensor}, i.e. $G/H$ is
a metric $G$-$\ANE$-space (related definitions are in Section
2).}. The following theorem,
proved in \cite[Theorem 9]{AgRep},
asserts
in favor of this conjecture: if $\Bbb X$ is a metric
$G$-$\ANE$-space, then each $G$-subspace $\Bbb Y\subset\Bbb X$
containing the bundle $\Bbb X_{\mathcal E}$ of extensor orbit
types is a $G$-$\ANE$.

We cannot omit the saturation condition from
Theorem~\ref{one1411}
(since there exists a 2-dimensional compact counterexample), but
we do have a pleasant (and important) exception for
$(\Sigma,d)$-universal (in the sense of Palais \cite[p.59]{Pal})
$G$-spaces. Until recently the solution of Palais problem on
existence of universal $G$-spaces was known only for finite
collection $\Sigma\subset\op{Orb}_G$ of orbit types and finite
dimension $d<\infty$ \cite[2.6]{Pal}; for finite dimension $d$
\cite{AgIzv}. The final solution of Palais problem (without any
restrictions on dimension $d$ and collection $\Sigma$) was
obtained in \cite{Ag}: the equivariant Hilbert space $\Bbb L_2$ is
an $(\op{Orb}_G,\infty)$-universal $G$-space. The following result
is a cornerstone of the theory of such universal $G$-spaces for
which we prefer alternative term -- an {\it isovariant absolute
extensor}, $\op{Isov-AE}$.
    \begin{thm}\label{one15}
 Let  $G$ be a compact Lie group, ${\cal C}\subset\op{Orb}_G$ a family of orbit types and
$\Bbb X$ an isovariant absolute extensor. Then the bundle $\Bbb
X_{\cal C}\subset\Bbb X$ of orbit type ${\cal C}$ is $G$-$\ANE$.
   \end{thm}

  \smallskip
The consequences of this theorem and another results of the theory
of isovariant absolute extensors will be presented in the
subsequent publications of the first author. The proofs of
Theorems \ref{one1411} and \ref{one15} will be based on the {\it
problem of extending the action of groups} which was first posed
by Shchepin in view of its connection with the problem of
extending equivariant maps (see \cite{AgRep}). The diagram ${\cal
D}=\{\Bbb
X\mathop{\rightarrow}\limits^{p}X\mathop{\hookrightarrow}\limits^{i}
Y\}$, in which the $G$-space $\Bbb X$  has the orbit type ${\cal
C}\subset\op{Orb}_G$, $p\colon\Bbb X\rightarrow X$ is an orbit
projection and $i$ is a closed topological embedding of the orbit
space $X$ into a space $Y$, will be called {\it ${\cal
C}$-admissible}. We say that the {\it problem of extending the
action is solvable for the ${\cal C}$-admissible diagram $\cal
D$}, provided that there exists an equivariant embedding $j:\Bbb
X\hookrightarrow\Bbb Y$ into a $G$-space $\Bbb Y$ of orbit type
${\cal C}$ (called a {\it ${\cal C}$-solution} of the problem of
extending the action for given diagram) covering $i$, i.e. the
embedding $\tilde{j}:X\hookrightarrow p(\Bbb Y)$ of orbit spaces
induced by $j$ coincides with $i$. Note that this definition
implies that the embedding $j$  is closed and $p(\Bbb Y)=Y$. If
the family of orbit types ${\cal C}$ coincides with $\op{Orb}_G$,
then the notation ${\cal C}$ is omitted.

\smallskip We say that  the problem of extending the action
(denoted briefly by PEA) is {\it solvable for the class $\cal F$ of
spaces}, if for each admissible diagram $\cal D$ in which  $\Bbb X,
X$ and $Y$ belong to $\cal F$ there exists a solution of the PEA
for which $\Bbb Y\in{\cal F}$. For compact group $G$, the PEA is
solvable for  the class of stratifiable spaces  (see
\cite{AgRep}).  Here for  the class of metric spaces we supplement
this result with an information on the $G$-orbit type of a
solution  of the PEA.
 \begin{thm}\label{one}  Let
$G$ be a compact Lie group, ${\cal C}\subset\op{Orb}_G$ a family
of orbit types with $(G)\in {\cal C}$ and ${\cal D}=(\Bbb
X\mathop{\rightarrow}\limits^{p}X\mathop{\hookrightarrow}\limits^{i}
Y)$ a metric ${\cal C}$-admissible diagram. Then for each solution
$s\colon\Bbb X\hookrightarrow\Bbb Y$ of the $\op{PEA}$ for ${\cal
D}$ there exists a metric ${\cal C}$-solution $s_1\colon\Bbb
X\hookrightarrow\Bbb Y_1$ of the $\op{PEA}$ for ${\cal D}$
majorized by  $s$, $s\ge s_1$ \footnote{ Recall that {\it $s$
majorizes $s_1$} if there exists a $G$-map $\mps{h}{\Bbb Y}{\Bbb
Y_1}$ such that $h\circ s=s_1$, $h\restriction_{\Bbb
X}=\op{Id}_{\Bbb X}$ and $\tilde h=\op{Id}_{Y}$. }.
  \end{thm}
We show in Section 3 that Theorem \ref{one} implies the validity
of Theorems  \ref{one1411} and \ref{one15}.

\section{Preliminary facts and results}

 In what follows we shall assume all spaces
 (resp. maps) to be metric (resp. continuous),
if they do not arise as a result of some constructions. For
$A\subset X$ we use standard notations: $\op{Cl}A$  -- for the
closure; $\op{Int}A$ -- for the interior. We use the notation
$f\restriction_{A}$ for the restriction of map $\mps{f}{X}{Y}$ on
$A\subset X$, or simply $f\restriction$, provided it is clear a
set to which we are referring. Since $f$ is an extension of
$f\restriction_{A}$, we denote this as
$f=\op{ext}(f\restriction_{A})$.

In what follows $G$ will be a compact group. An {\it action of $G$
on a space $X$} is a homomorphism $T:G\to\op{Aut} X$ of $G$ into
the group $\op{Aut} X$ of all autohomeomorphisms of $X$ such that
the map $G\times X\to X$ given by $(g,x)\mapsto T(g)(x)=g\cdot x$
is continuous. A space $X$ with a fixed action of $G$ is called a
{\it $G$-space}.

For any point $x\in X$, $G_{x} =\{g\in G\ |\ g\cdot x=x \}$ is a
closed subgroup of $G$ called  the {\it isotropy subgroup} of $x$;
$G(x)=\{g\cdot x\mid g\in G\}\subset X$ is called  the {\it orbit}
of $x\in X$. The set of all orbits is denoted by $X/G$ and the
natural map $p=p_{X}:X\to X/G$, given by $p(x)=G(x)$, is called
the {\it orbit projection}. We call the set $X/G$ of all orbits
equipped with the quotient topology induced by $p$ the {\it orbit
space of $X$} (see \cite{Br} for more details on compact
transformation groups). In what follows we shall denote $G$-spaces
and their orbit spaces as follows: $\Bbb X,\Bbb Y,\Bbb Z,\dots $
for $G$-spaces, and $X,Y,Z,\dots$ for their orbit spaces.

The map $\mps f{\Bbb X}{\Bbb Y}$ of $G$-spaces is called {\it
equivariant} or a {\it $G$-map}, if $f(g\cdot x)=g\cdot f(x)$ for
all  $g\in G$ and $x\in\Bbb X$. Each $G$-map $\mps f{\Bbb X}{\Bbb
Y}$ induces a map $\mps {\tilde f}{X}{Y}$ of orbit spaces by the
formula $\tilde f(G(x))=G(f(x))$. We call an equivariant
homeomorphism an {\it equimorphism}. The equivariant map $\mps
f{\Bbb X}{\Bbb Y}$ is said to be  {\it isovariant} if
$G_x=G_{f(x)}$ for all $x\in\Bbb X$.

The isovariant map $\mps f{\Bbb X}{\Bbb Y}$ is said to be an {\it
isogeny} if its induced map $\mps{\tilde f}{X}{Y}$ is a
homeomorphism. More generally, the equivariant map $f$ is said to
be an {\it equigeny} if $\tilde f$ is a homeomorphism \footnote{We
remark that {\it equigeny} is a new term and it is close to {\it
isogeny}, a term used by Palais in \cite[p. 12]{Pal}.}. By
\cite[3.7.10]{Eng} each equigeny and isogeny is perfect.

Observe that all $G$-spaces and $G$-maps generate a category
denoted by  $G$-$\op{TOP}$ or $\op{EQUIV}$-$\op{TOP}$, provided
that no confusion occurs. If  $"***"$ is any notion from
non-equivariant topology, then $"G$-$***"$ or $"\op{Equiv}$-$***"$
means the corresponding equivariant analogue.

The subset  $A\subset\Bbb X$ is called {\it invariant} or a {\it
$G$-subset}, if   $G\cdot \Bbb A=\Bbb A$. For each  closed
subgroup $H<G$ (in what follows this sign will be used for closed
subgroups; a normal closed subgroup is denoted as $H\triangleleft
G$) we introduce the following sets: $\Bbb X^H = \{x\in\Bbb X \mid
H\cdot x = x\}$ (which is called an {\it $H$-fixed set})  and
$\Bbb X_{H} = \{x\in \Bbb X\mid G_x=H\}$. It is clear that $\Bbb
X^{(H)} = \bigcup\{\Bbb X_K \mid K<G \ \text{and}\ H'<K\ \text{for
some conjugated subgroup}\ H'\sim H\}$ coincides with $G\cdot \Bbb
X^H$ and $\Bbb X_{(H)}\mean\bigcup\{\Bbb X_K \mid \ K<G\ \hbox{
conjugates with}\ H\}$ coincides with $G\cdot \Bbb X_H$.

Let $\op{Conj}_G$ be the set of all conjugated classes of closed
subgroups of $G$ and $\op{Orb}_G$ a collection of all homogeneous
spaces up to equimorphisms. We endow these sets with the following
partial orders:  $(K)\le (H)\Leftarrow\Rightarrow K\ \text{is
contained in}\ H'\in(H)$; \ $G/K\ge G/H\Leftarrow\Rightarrow\ $
there exists an  equivariant map $\mps{f}{G/K}{G/H}$. It is
evident that the bijection $(H)\in\op{Conj}_G\mapsto
G/H\in\op{Orb}_G$ inverses this order. In view of this we identify
these sets, provided that no confusion occurs, and we shall use
the unified term -- the {\it set of $G$-orbit types} and the
unified notation -- $\op{Orb}_G$.

We denote the family $\{(G_x)\mid x\in\Bbb X \}\subset\op{Orb}_G$
of {\it orbit types} of $\Bbb X$  by $\op{type}(\Bbb X)$. If
${\cal C}\subset\op{Orb}_G$, then $\Bbb X_{\cal C}\mean\{x\mid
(G_x)\in {\cal C}\}\subset\Bbb X$ -- {\it the bundle of orbit
types ${\cal C}$}; $\Bbb X^{\mathcal C}\mean\{x\mid (G_x)\ge (H)\
\text{for some}\ (H)\in {\mathcal C}\}\subset\Bbb X$ (here and
throughout the paper the sign $\mean$ is used for the introduction
of the new objects placed to the left of it). The $G$-space $\Bbb
X$ {\it has an orbit type  ${\cal C}$} if $\op{type}(\x)\subset
{\cal C}$ or $\Bbb X=\Bbb X_{\cal C}$. The family ${\cal
C}\subset\op{Orb}_G$ is said to be {\it closed} if $(H)\in{\cal
C}$ as soon as $(H)\ge(K)$ for some $(K)\in{\cal C}$; the family
${\cal C}\subset\op{Orb}_G$ is called {\it saturated} if
$(K)\ge(L)\ge(H)$ for some $(K),(H)\in{\cal C}$ implies
$(L)\in{\cal C}$. Also we say that the subfamily ${\cal F}$ of
${\cal C}\subset\op{Orb}_G$ is cofinal if for each $(H)\in{\cal
C}$ there exists $(K)\in{\cal F}$ with $(H)\le (K)$.

If $K<G$ and ${\cal C}\subset\op{Orb}_G$, we set  ${\cal
C}\restriction_K\mean\{(H)\in\op{Orb}_K \mid H<K \ \text{and }\
G/H\in{\cal C}\}$ and ${\cal C}'={\cal C}\cup\{(G)\}$. It is clear
that: $\Bbb X_{{\cal C}}$ is open in $\Bbb X_{{\cal C}'}$; if
$(K)\in{\cal C}$, then ${\cal C}\restriction_K$ coincides with
$({\cal C}\restriction_K)'$. It is clear also that  if ${\cal
C}\subset\op{Orb}_G$ is  closed family and ${\cal
C}\restriction_K\not=\emptyset$, then ${\cal C}\restriction_K$ is
closed in  $\op{Orb}_K$.

We now introduce several concepts related to extension of
equivariant and isovariant maps  partially defined in metric
$G$-spaces. A space $\Bbb X$ is called an {\it equivariant
absolute neighborhood extensor}, $\Bbb X\in\op G$-$\op{ANE}$, if
each $G$-map $\mps{\varphi}{\Bbb A}{\Bbb X}$ defined on a closed
$G$-subset $\Bbb A\subset \Bbb Z$ of metric $G$-space $\Bbb Z$ and
called the {\it partial $G$-map} can be $G$-extended onto a
$G$-neighborhood $\Bbb U\subset \Bbb Z$ of $\Bbb A$,
$\mps{\hat\varphi}{\Bbb U}{\Bbb X}, \hat\varphi\restriction_{\Bbb
A}=\varphi$. A space $\Bbb X$ is called an  {\it isovariant
absolute neighborhood extensor}, $\Bbb X\in\op{Isov}$-$\op{ANE}$,
if each partial isovariant map $\pam{\Bbb Z}{\Bbb A}{\varphi}{\Bbb
X}$ can be isovariantly extended onto a $G$-neighborhood $\Bbb
U\subset \Bbb Z$ of $\Bbb A$. If $\varphi$ can be extended in
$\Bbb U=\Bbb Z$, then $\Bbb X$ is called an {\it equivariant
absolute extensor} ($\Bbb X\in\op G$-$\op{AE}$) in the equivariant
case  or an {\it isovariant absolute extensor} ($\Bbb
X\in\op{Isov}$-$\op{AE}$) in the isovariant case. If the acting
group $G$ is trivial, then these notions are transformed into the
notions of absolute [neighborhood] extensors  -- $\op{A[N]E}$.

 The following examples of $\op
G$-$\op{AE}$-spaces are well-known: each Banach $G$-space (see
\cite[p. 117]{Matu} and \cite[p. 155]{Ab}); each linear normed
$G$-space for compact Lie group $G$ (see \cite[p. 488]{Mu}). We
shall depend heavily on the Slice theorem \cite{Br} which we
prefer to formulate as follows: $G/H\in G$-$\op{ANE}$ for each
closed subgroup $H$ of compact Lie group $G$.

 \smallskip
   \begin{defi} \label{aslice2+7} A closed subgroup $H<G$ of compact
group $G$ is called an extensor subgroup if one of the following
equivalent properties holds:
\begin{enumerate}
   \item[$(1)$] $G/H$ is finite-dimensional and locally  connected;
 \item[$(2)$] there exists a normal subgroup
$P\triangleleft G$ such that  $P<H$ and $G/P$ is a compact Lie
group;
 \item[$(3)$] $G/H$ is a topological manifold; or
 \item[$(4)$] $G/H$ is a metric $G$-$\ANE$-space.
\end{enumerate}
 \end{defi}
 The  equivalence of the first three properties was  proved in \cite{Ptr}; for
the proof of $(2)\equiv(4)$ see \cite{AU}. The last property
justifies the name of the term. The following theorem on
approximate slice of a $G$-space was proved in \cite{AU}.
     \begin{thm} \label{slice2-6} Let a compact group $G$ act on a $G$-space
$\Bbb X$. Then for each neighborhood ${\mathcal O}(x)$ of
$x\in\Bbb X$ there exist a neighborhood $\V=\V(e)$ of the unit
$e\in G$, an extensor subgroup $K<G,G_x<K$, and a slice map
$\mps{\alpha}{\Bbb U}{G/K}$ where $\Bbb U$ is an invariant
neighborhood of $x$ such that $x\in\alpha\ii({\mathcal V}\cdot
[K])\subset{\mathcal O}(x)$.
 \end{thm}

The orbit type $(H)$ is called  {\it extensor} if $H<G$ is an
extensor subgroup. The collection of all {\it extensor orbit
types} is denoted by ${\cal E}$. We say that $G$-subspace $\Bbb
Y\subset\Bbb X$ is   {\it $G$-dense} if $\Bbb Y^H\subset\Bbb X^H$
is  dense for each   subgroup $H<G$. It was proved in \cite{AU}
with the help of Theorem \ref{slice2-6}
 that
 \begin{enumerate}\setcounter{enumi}{4}
 \item $\Bbb X_{\mathcal E}\subset\Bbb X$ is   $G$-dense if and
only if $\Bbb X$ is an equivariant neighborhood extensor  for
metric $G$-spaces with zero-dimensional orbit spaces, $\Bbb X\in
G$-$\ANE(0)$; and
 \item a linear normed $G$-space $\Bbb L$
is a $G$-$\op{AE}$ if and only if $\Bbb L_{\mathcal E}\subset\Bbb
L$ is $G$-dense (equivariant Dugunji's theorem).
\end{enumerate}
 There exists an example of a linear normed $G$-space $\Bbb L\not\in G$-$\op{AE}$
for which $\Bbb L_{\mathcal E}\subset\Bbb L$ is  dense (but not
$G$-dense).

\medskip The proof  of the following Palais Metatheorem
\cite{Pal} is based on the stabilization of nested sequence  of
compact Lie groups.
 \begin{prop} \label{t5.9++} Let ${\mathcal P}(H)$
be a property  which depends on compact Lie group  $H$. Suppose
that ${\mathcal P}(H)$ is true, provided ${\mathcal P}(K)$ is true
for each compact Lie group $K$ isomorphic to a proper subgroup of
$H$. If ${\mathcal P}(H)$ is true for trivial group $H=\{e\}$,
then ${\mathcal P}(H)$ is true for all compact Lie groups $H$.
 \end{prop}

 If there is no danger of ambiguity, we shall
omit definitions of some notions, which arise in a natural manner.
As a rule, this remark concerns also the assertions analogous to
the proved ones.

\section{Reduction of  Theorems  \ref{one1411} and \ref{one15} to
Theorem \ref{one}}
 Theorem \ref{one} will be convenient for our aims in the following detailed form:
  \begin{prop}\label{one17}
Let $G$ be a compact Lie group, ${\cal C}\subset\op{Orb}_G$ a
family of orbit types with $(G)\in {\cal C}$. If a metric
$G$-space $\Bbb Y$ contains a $G$-subspace   $\Bbb X$ of orbit
type ${\cal C}$ as a closed subset, then $\Bbb X$ is contained in
a metric $G$-space $\Bbb Y_1$ of orbit type ${\cal C}$ as a closed
subset, and  there exists a $G$-map $\mps{h}{\Bbb Y}{\Bbb Y_1}$
such that $h\restriction_{\Bbb X}=\Id_{\Bbb X}$.
   \end{prop}
Now the {\bf proof of Theorem \ref{one1411}} easily follows  from
the following lemmata.
   \begin{lem}\label{88}
 Let $G$ be a compact Lie group, $\Bbb X\in G$-$\op{A[N]E}$ and ${\cal C}\subset\op{Orb}_G$
a closed  family. Then $\Bbb X_{\cal C}\subset\Bbb X$ is
$G$-$\op{A[N]E}$.
   \end{lem}
 \begin{proof}  We consider a closed  $G$-embedding of  $\Bbb X_{\cal C}$
into a metric $G$-space  $\Bbb Y$. In view of  Proposition
\ref{one17}, there exist  a closed  $G$-embedding  $\Bbb X_{\cal
C}\hookrightarrow\Bbb Y_1$ into a metric  $G$-space  $\Bbb Y_1$ of
orbit  type  ${\cal C}$  and a $G$-map $\mps{h}{\Bbb Y}{\Bbb Y_1}$
such that  $h\restriction_{\Bbb X_{\cal C}}=\Id_{\Bbb X_{\cal
C}}$. If $\Bbb X\in G$-$\op{AE}$, then there exists a $G$-map
$\mps{r}{\Bbb Y_1}{\Bbb X},r\restriction_{\Bbb X_{\cal C}}=\Id$.
Since ${\cal C}$ is a closed family and $\op{type}(\Bbb
Y_1)\subset{\cal C}$, $\op{type}(r(\Bbb Y_1))\subset \cal C$.
Therefore $\op{Im} r\subset \Bbb X_{\cal C}$ and $r\circ h$ is the
desired $G$-retraction. The case of $G$-$\op{ANE}$-space is proved
analogously.
 \end{proof}
  \begin{lem}\label{98}
 Let  $G$ be a compact Lie group and ${\cal C}\subset\op{Orb}_G$ a saturated
family. Then $\Bbb X_{\cal C}$ is open in $\Bbb X_{\cal F}$ where
${\cal F}\mean\cup\{{\cal C}^{(K)}\mid (K)\in{\cal C}\}$ (the
so-called closed hull of ${\cal C}$).
   \end{lem}
 \begin{proof} Let $(G_x)\in{\cal C}$. There exists a neighborhood  ${\cal
O}(x)\subset\Bbb X_{\cal F}$ such that $(G_y)\le(H)$ for each
$y\in{\cal O}(x)$. Since ${\cal F}$ is the closed hull of ${\cal
C}$, there exists $(K)\in{\cal C},(K)\le (G_y)$. Since ${\cal C}$
is saturated, $(K)\le(G_y)\le(H)$ implies that $(G_y)\in{\cal C}$.
 \end{proof}

\medskip For the proof of {\bf Theorem \ref{one15}} we consider
a  closed $G$-embedding $\Bbb X_{\cal C}$ into a metric $G$-space
$\Bbb Y$. First we assume that $(G)\in{\cal C}$. By Proposition
\ref{one17} there exists a closed $G$-embedding $\Bbb X_{\cal C}$
into a metric $G$-space $\Bbb Y_1$ of orbit type  ${\cal C}$ and a
$G$-map $\mps{h}{\Bbb Y}{\Bbb Y_1}$ such that $h\restriction_{\Bbb
X_{\cal C}}=\Id_{\Bbb X_{\cal C}}$. Since $\Bbb X$ is an
isovariant absolute extensor, there exists a $G$-map $\mps{r}{\Bbb
Y_1}{\Bbb X}$ isovariant on the complement such that
$r\restriction_{\Bbb X_{\cal C}}=\Id_{\Bbb X_{\cal C}}$. Since
$r(\Bbb Y_1)\subset(\Bbb X)_{\cal C}$, $r\circ h$  is the desired
$G$-retraction.

In the general case, $\Bbb X_{\cal C'}\in\op G$-$\op{AE}$ where
${\cal C}'={\cal C}\cup\{(G)\}$. Since $\Bbb X_{\cal C}$ is open
in $\Bbb X_{\cal C'}$, the proof is completed.

\section{Proof of Theorem  \ref{one}}
Let $G$ be a compact Lie group, ${\cal C}\subset\op{Orb}_G$ a
family of orbit types with $(G)\in {\cal C}$ and $\mps{h}{\Bbb
X}{\Bbb X'}$ an equigeny with $\op{type}(\Bbb X')\subset{\cal C}$.
We interest for $G$-embedding $\Bbb X\hookrightarrow\Bbb Y$
whether there exist a $G$-embedding $\Bbb X'\hookrightarrow\Bbb
Y'$ with $\op{type}(\Bbb Y')\subset{\cal C}$ and an equigeny
$\mps{H}{\Bbb Y}{\Bbb Y'}$  extending $h$. We say that $H$ {\it
${\cal C}$-solves the problem of extending of an equigeny for
${\cal C}$-admissible diagram ${\cal D}=\{\Bbb
Y\matop{\hookleftarrow}{}{}\Bbb
X\mathop{\rightarrow}\limits^{h}\Bbb X'\}$}. Since $\mps{\Id}{\Bbb
X}{\Bbb X}$ is an equigeny, Theorem \ref{one} is reduced to more
general assertion:
  \begin{thm}\label{three+++}
 If $\Bbb X$ is closed in $\Bbb Y$, then the problem of extending of an equigeny for
${\cal D}$ is ${\cal C}$-solved, i.e. there exist a closed
$G$-embedding $\Bbb X'\hookrightarrow\Bbb Y'$ with $\op{type}(\Bbb
Y')\subset{\cal C}$ and an equigeny  $\mps{H}{\Bbb Y}{\Bbb Y'}$
extending $h$.
 \end{thm}
   \begin{lem}\label{three+}
 If $\Bbb X$ is open in $\Bbb Y$, then the problem of extending of equigeny for
${\cal D}$ is ${\cal C}$-solved.
 \end{lem}
\begin{proof}
 Let $\Bbb Y'\mean\Bbb Y/\sim$ be the quotient space generated by
the equivalence  $y\sim x\in\Bbb X$ iff $y\in\x$ and $h(x)=h(y)$;
$y\sim y_1\in\Bbb Y\setminus\Bbb X$ iff $y\in\Bbb Y\setminus\Bbb
X$ and $G(y)=G(y_1)$. Then the desired equigeny $\mps{H}{\Bbb
Y}{\Bbb Y'=\Bbb Y/\sim}$ extending $h$ is the quotient map.
 \end{proof}

The proof of the following result is based on Lemma \ref{three+}
and follows parallel to \cite[Lemma 8]{AgRep}. It  reduces the
proof of Theorem \ref{three+++} to the case when the space $\Bbb
X$ has no fixed points.
 \begin{lem}\label{four+-} The validity of Theorem \ref{three+++} for
all $\cal C$-admissible diagrams with $\Bbb X^G=\varnothing$
implies  its validity for all  $\cal C$-admissible diagrams.
 \end{lem}
\begin{proof}
 Let $\Bbb F\mean\Bbb X^G$. By hypotheses the equigeny
 $\mps{h\restriction}{\x\setminus\Bbb F}{\Bbb X'\setminus\Bbb F}$
 can be extended up to an equigeny
 $\mps{\eta}{\Bbb Y\setminus\Bbb F}{\Bbb Z'}$. Next we
apply  Lemma \ref{three+} to $\eta$ and the open embedding $\Bbb
Y\setminus\Bbb F\hookrightarrow\Bbb Y$.
 \end{proof}

\section{Proof of Theorem  \ref{three+++}}
In view of  Lemma \ref{four+-} we can assume that the $G$-space
$\Bbb X$ has no $G$-fixed  points, $\Bbb X^G=\varnothing$. When
such is the case  it is sufficient by Lemma \ref{three+} for some
$G$-neighborhood $\Bbb Z,\Bbb X\subset\Bbb Z\subset\Bbb Y$, to
construct a $G$-embedding $\Bbb X'\hookrightarrow\Bbb Z'$ with
$\op{type}(\Bbb Z')\subset{\cal C}$ 
and  an equigeny
$\mps{H_1}{\Bbb Z}{\Bbb Z'}$ 
extending $h$.

In what follows the argument  will be carried out by induction on compact Lie group
$G$ based on Palais Metatheorem  \ref{t5.9++}. If $|G|=1$, then
the situation under consideration is trivial. Now we suppose that
for each proper subgroup $K<G$ Theorem \ref{three+++} has been
proved and let us show its validity in  case of the $G$-action.
First we consider a special case:
  \begin{lem}\label{five-1+}  If $\Bbb X'$ admits a
nontrivial slice map  $\mps{\psi}{\Bbb X'}{G/K}$ with $(K)\in{\cal
C}$, then  Theorem \ref{three+++} is valid for the $\cal
C$-admissible diagram  ${\cal D}=\{\Bbb
Y\matop{\hookleftarrow}{}{}\Bbb
X\mathop{\rightarrow}\limits^{h}\Bbb X'\}$.
 \end{lem}
 Before the proof we recall the notion of a twisted product. Let us consider a compact
group $G$, a metric $H$-space $\Bbb S$ where $H<G$ and the
diagonal action of $H$ on the product $G\times\Bbb S$ defined as
$h\cdot (g,y)\mean (g\cdot h^{-1},h\cdot y)$. By $[g,y]$ we denote
the element $H\cdot (g,y)=\{(g\cdot h^{-1},h\cdot y)\ \mid h\in
H\}$ of the orbit space $(G\times\Bbb S)/H$. It turns out that the
formula $g_1\cdot [g,y]=[g_1\cdot g,y]$ where  $g,g_1\in G,y\in
\Bbb S$, correctly define the continuous action of $G$ on the
orbit space $(G\times\Bbb S)/H$ called a {\it twisted product}
(and denoted as $G\times_H\Bbb S$).

The notion of the twisted product arise naturally in studies of a
$G$-space admitting the slice map, say, $\mps{\varphi}{\x}{G/H}$,
$H<G$. Then  $\x$ can be identified with the twisted product
$G\times_H\Bbb S$ where $\Bbb S=\varphi\ii([H])$ is a $H$-slice:
$[g,s]\in G\times_H\Bbb S\leftrightarrow x\in\x$ (see this and
another properties of twisted products in \cite{Br}).
\begin{proof}
 Since  $G/K\in G$-$\ANE$, there exists a $G$-extension
$\mps{\tilde\varphi}{\Bbb U}{G/K}$  of $\mps{\varphi\mean\psi\circ
h}{\Bbb X}{G/K}$ defined on some $G$-neighborhood $\Bbb U, \Bbb
X\subset \Bbb U\subset \Bbb Y$. It is clear that
$\tilde\varphi\ii[K]\supset\varphi\ii[K]$ and $\psi\ii[K]$ are
$K$-spaces for proper compact subgroup $K<G$ with

  \centerline{$\op{type}(\psi\ii[K])\subset{\cal C}\restriction_K\mean\{(H)\mid
H<K \ \text{and }\ G/H\in{\cal C}\}\subset\op{Orb}_K$.}

\noindent Since $(K)\in{\cal C}\restriction_K$ and
$\mps{h\restriction}{\varphi\ii[K]}{\psi\ii[K]}$ is an equigeny,
there exist  by inductive hypothesis an equigeny
$\mps{H'}{\tilde\varphi\ii[K]}{\Bbb W'\supset\psi\ii[K]}$
extending $h\restriction$.

Let us consider the following commutative square diagram:
$$
\begin{array}{ccc}
G\times_K\varphi\ii[K]=\Bbb X& \stackrel{}{\hookrightarrow} & G\times_K\tilde\varphi\ii[K]=\Bbb Y\\
\downarrow \Id\times_K h\restriction=h&& \downarrow H\mean\Id\times_K H'\\
G\times_K\psi\ii[K]=\Bbb X' & \stackrel{}{\hookrightarrow} & \Bbb
Y'\mean G\times_K\Bbb W'.
\end{array}
$$
Since $\mps{H=\Id\times H'}{\Bbb
Y=G\times_K\tilde\varphi\ii[K]}{\Bbb Y'=G\times_K\Bbb W'}$ is an
equigeny, $\op{type}(\Bbb Y')\subset{\cal C}\restriction_K$  and
${\cal C}\restriction_K$ lies naturally into ${\cal C}$, the proof
of Lemma is completed.
\end{proof}
 \medskip We continue the proof of Theorem  \ref{three+++}. The
following result is an easy consequence of the Slice theorem  and
hereditary paracompactness of $\Bbb Y$.
   \begin{lem}\label{ActExt3++}
 There exist a closed $G$-neighborhood  $\Bbb E,\Bbb X\subset \Bbb E\subset \Bbb Y$, and
its local-finite $G$-cover $\sigma\in\cov \Bbb E$ consisting of
closed $G$-subsets $\{\Bbb E_\gamma\subset \Bbb
E\}_{\gamma\in\Gamma}$ such that for each $\gamma\in\Gamma$, $\Bbb
E_\gamma\cap \Bbb X\not=\emptyset$  and the $G$-space $\Bbb
V_\gamma \mean h( \Bbb E_\gamma)$ admits a nontrivial  slice map
$\mps{\alpha}{\Bbb V_{\gamma}}{G/K}$ for some $(K)\in{\cal C}$.
\end{lem}
 \smallskip Because of our aim -- to extend $h$ up to an equigeny defined
on a $G$-neighborhood $\Bbb Z,\Bbb X\hookrightarrow\Bbb
Z\subset\Bbb Y$, in what follows we can certainly assume that
$\Bbb E=\Bbb Y$.

To have a possibility to argue by a new {\it transfinite }
induction, we well-order the set $\Gamma$ indexing the elements of
family $\{ \Bbb E_\gamma\}$. Without loss of generality we can
assume that $\Gamma$ has the maximal element  $\omega$ which is
not limit. We put $\Bbb Q_\gamma\mean \Bbb X\cup\cup\{ \Bbb
E_{\gamma'}\mid\gamma'<\gamma\}$ for each limit ordinal $\gamma$,
otherwise we set  $ \Bbb Q_\gamma\mean \Bbb X\cup\cup\{ \Bbb
E_{\gamma'}\mid\gamma'\le\gamma\}$.

It is obvious that  $\Bbb Y= \Bbb Q_\omega$ is the body of the
increasing system of closed  subsets $\{ \Bbb Q_\gamma\}$,
moreover $ \Bbb Q_{\gamma'}\cup \Bbb E_{\gamma}= \Bbb Q_\gamma$
for $\gamma=\gamma'+1$. As $\sigma\in\cov \Bbb Y$ is local finite,
the following property of $\Bbb Q_{\gamma}$ holds (see
\cite[Section 2 of Introduction]{Sp}):
  \begin{enumerate}\setcounter{enumi}{0}
    \item
if the ordinal $\gamma$ is limit, then $U\subset \Bbb Q_{\gamma}$
is open if and only if $(\Bbb Q_{\gamma'})\cap U$ is open in $\Bbb
Q_{\gamma'}$ for all $\gamma'<\gamma$ (or equivalently, $\Bbb
Q_{\gamma}$ coincides with the limit
$\matop{\lim}{\rightarrow}{}\{\Bbb Q_{\gamma'}\mid
\gamma'<\gamma\}$ of the direct spectrum).
  \end{enumerate}
 Before proceeding further, we generate some notations.
For each $\gamma\in\Gamma$ we choose a closed neighborhood $\Bbb
P_\gamma,\x\subset\Bbb P_\gamma\subset \Bbb Q_{\gamma}$, such that
 \begin{enumerate}\setcounter{enumi}{1}
 \item  $\Bbb P_{\gamma'}\subset\Bbb P_{\gamma}$ for all $\gamma'<\gamma$;
 \item $\Bbb P_{\gamma+1}\setminus\Bbb
P_{\gamma}\subset\Bbb E_{\gamma}$.
  \end{enumerate}
 In particular, the
closed $G$-neighborhood $\Bbb Z\mean\Bbb P_\omega$ of $\Bbb X$ in
$\Bbb Y$ is the limit $\matop{\lim}{\rightarrow}{}\{\Bbb
P_{\gamma'}\mid \gamma'<\omega\}$ of the direct spectrum. Since of
$(3)$ and local finiteness of $\{\Bbb E_\gamma\}$ we have
 \begin{enumerate}\setcounter{enumi}{3}
 \item  for each limit ordinal $\gamma$, $\Bbb P_{\gamma}$ coincides with
the direct limit $\matop{\lim}{\rightarrow}{}\{\Bbb
P_{\gamma'}\mid \gamma'<\gamma\}$ .
  \end{enumerate}

By transfinite induction  we specify closed neighborhoods $\{\Bbb
P_\gamma\}$ and construct for each $\gamma\in\Gamma$ a closed
$G$-embedding $\Bbb X'\hookrightarrow\Bbb P_{\gamma}'$ and an
equigeny $\mps{H_{\gamma}}{\Bbb P_{\gamma}}{\Bbb P_{\gamma}'}$
extending $h$ such that
\begin{enumerate}
 \item[$(5)_{\gamma}$] $\Bbb P'_{\gamma_1}\subset \Bbb
P'_{\gamma}$ and $H_{\gamma}\restriction_{\Bbb
P_{\gamma_1}}=H_{\gamma_1}$ for all  $\gamma_1<\gamma$.
  \end{enumerate}
We set $\Bbb Z'\mean \Bbb P_\omega'$. It follows by $(5)$ that for
limit ordinal $\gamma$, $\Bbb P_{\gamma}'$ is the limit
$\matop{\lim}{\rightarrow}{}\{\Bbb P_{\gamma_1}'\mid
\gamma_1<\gamma\}$ of the direct spectrum
 (in particular, $\Bbb
Z'=\matop{\lim}{\rightarrow}{}\{\Bbb P_{\gamma}'\mid
\gamma<\omega\}$) and the continuous map
$\mps{H=\matop{\lim}{\rightarrow}{}\{
H_\gamma\}}{\matop{\lim}{\rightarrow}{}\{\Bbb P_{\gamma}\mid
\gamma<\omega\}}{\Bbb Z'=\matop{\lim}{\rightarrow}{}\{\Bbb
P_{\gamma}'\mid \gamma<\omega\}}$ of limits  of the direct spectra
is the equigeny. Since $\Bbb Z\subset\Bbb Y$ is the limit
$\matop{\lim}{\rightarrow}{}\{\Bbb P_{\gamma}\mid \gamma<\omega\}$
of the direct spectrum, $H$ is the required equigeny extending
$h$, that leads to the completion of the {\bf proof of Theorem
\ref{three+++}}.

\medskip Let a topological  space  $(D,\tau_D)$ represent as a union $A\cup
B$ of its subspaces. We consider a {\it weak topology $\tau_w$ on
$D$, generated by  $A$ and $B$}: $U\in\tau_w$ if and only if
$A\cap U\subset A$ and $B\cap U\subset B$ are open. The subspaces
{\it $A$ and $B$ generate the topology of $D$} if the weak
topology $\tau_w$ coincides with  $\tau_D$. It is known \cite{Sp}
that
 \begin{enumerate}
  \item[$(6)$] the subspaces $A$ and $B$ generate the topology of $D$ in the case
that $A$ and $B$ are closed in $D$.
  \end{enumerate}

 The {\bf base} of the inductive argument is easily established
with the help of Lemma \ref{five-1+}. Let $\gamma_0\in\Gamma$ be a
minimal ordinal.
 \begin{lem}\label{five+} There exist a closed
$G$-embedding $\Bbb X'\hookrightarrow\Bbb P_{\gamma_0}'$ and an
equigeny $\mps{H_{\gamma_0}}{\Bbb P_{\gamma_0}}{\Bbb
P_{\gamma_0}'}$ extending $h$.
 \end{lem}
\begin{proof} Since  the $G$-space  $\Bbb V_{\gamma_0}=h(\Bbb E_\gamma\cap\x)$ admits  a nontrivial
slice map, Lemma \ref{five-1+} implies the existence of a closed
$G$-embedding $\Bbb V\hookrightarrow\Bbb E_{\gamma_0}'$ and an
equigeny $\mps{\eta}{\Bbb E_{\gamma_0}}{\Bbb E_{\gamma_0}'}$
extending $h$.

Let $\Bbb P_{\gamma_0}\mean\Bbb X\cup\Bbb E_{\gamma_0}\subset\Bbb
Y$. Since $\Bbb X$ and $\Bbb E_{\gamma_0}$ are simultaneously
closed in $\Bbb P_{\gamma_0}$, the topology of $\Bbb P_{\gamma_0}$
coincides with a weak topology generated by $\Bbb X$ and $\Bbb
E_{\gamma_0}$. Let us consider a union $\Bbb X'\cup_{\Bbb
V_{\gamma_0}}\Bbb E_{\gamma_0}'$ with a weak topology which we
denote by $\Bbb P_{\gamma_0}'$. It is now well understood that the
$G$-map $\mps{H_{\gamma_0}\mean\Id\cup\eta}{\Bbb P_{\gamma_0}=\Bbb
X\cup\Bbb E_{\gamma_0}}{\Bbb P_{\gamma_0}'=\Bbb X'\cup\Bbb
E_{\gamma_0}}$ defined as $\Id$ on $\Bbb X$ and $\eta$ on $\Bbb
E_{\gamma_0}$ is required.
\end{proof}

 \medskip The {\bf inductive step} consists of the following
 proposition. Let $\gamma\in\Gamma$ be successive for
 $\gamma_1\in\Gamma$, i.e. $\gamma=\gamma_1+1$.
 \begin{lem}\label{five-+1}
  There exist a closed $G$-embedding $\Bbb X'\hookrightarrow\Bbb
P_{\gamma}'$ and an equigeny $\mps{H_{\gamma}}{\Bbb
P_{\gamma}}{\Bbb P_{\gamma}'}$ extending $h$ such that the
conditions $(5)_{\gamma}$ holds.
   \end{lem}
  \begin{proof}
 By Lemma \ref{ActExt3++} $\Bbb V_{\gamma}$ admits  a nontrivial
slice map $\mps{\varphi}{\Bbb V_{\gamma}}{G/K},(K)\in{\cal C}$.
The Slice theorem ($\equiv$ $G/K\in G$-$\ANE$) implies that
 \begin{enumerate}\setcounter{enumi}{6}
 \item the partial $G$-map $\pam{\Bbb
P_{\gamma_1}'}{\Bbb V_{\gamma}}{\varphi}{G/K}$ can be extended up
to a slice map $\mps{\psi}{\Bbb U'}{G/K}$ defined into a closed
neighborhood $\Bbb U',  \Bbb V_{\gamma}\subset \Bbb U'\subset \Bbb
P_{\gamma_1}'$.
\end{enumerate}

 Let $\mps{\varphi\mean\psi\circ h_{\gamma_1}}{\Bbb
U}{G/K}$ be a slice map defined into $\Bbb U\mean
H_{\gamma_1}\ii(\Bbb U')$, $\hat {\Bbb U}\subset \Bbb
P_{\gamma_1}\cup \Bbb E_{\gamma}$ a closed neighborhood of $\Bbb
E_{\gamma}\cap\x$ such that $\hat {\Bbb U}\cap\Bbb
P_{\gamma_1}=\Bbb U$.

By Lemma \ref{five-1+} there  exist a closed $G$-embedding ${\Bbb
U}'\hookrightarrow\hat {\Bbb U}'$ and an equigeny
$\mps{\eta}{\hat {\Bbb U}}{\hat {\Bbb U}'}$ extending
$H_{\gamma_1}$.

 It is easy to check that the desired $G$-space $\Bbb
P'_{\gamma}$ is  the union  $\Bbb P'_{\gamma_1}\cup_{\Bbb U'} \hat
{\Bbb U}'$ endowed with a weak topology. It is evidently that for
$\Bbb P_{\gamma}\mean\Bbb P'_{\gamma_1}\cup \hat {\Bbb U}$ we have
$\Bbb P_{\gamma}\setminus\Bbb P_{\gamma_1}\subset\Bbb E_{\gamma}$.

The desired $G$-map $\mps{H_{\gamma}}{\Bbb P_{\gamma}=\Bbb
P'_{\gamma_1}\cup \hat {\Bbb U}}{ \Bbb P'_{\gamma}=\Bbb
P'_{\gamma_1}\cup _{\Bbb U'}\hat {\Bbb U}'}$ coincides with
$H_{\gamma_1}$ on $\Bbb P_{\gamma_1}$ and coincides with $\eta$ on
$\hat {\Bbb U}$. Since the topology of $\Bbb P_{\gamma}$ is
generated by $\Bbb P'_{\gamma_1}$ and $\hat {\Bbb U}$, $H_\gamma$
is the required equigeny  extending $H_{\gamma_1}$.
 \end{proof}

\bigskip Now let $\gamma\in\Gamma$ be a limit ordinal.
We consider the increasing family of constructed $G$-subspaces
$\{\Bbb P'_{\gamma'}\subset\Bbb
P'_{\gamma''}\}_{\gamma'\le\gamma''<\gamma}$ and the family of
equigenies $\{\mps{H_{\gamma'}}{\Bbb P_{\gamma'}}{ \Bbb
P'_{\gamma'}}\}_{\gamma'<\gamma} $. Furthermore, we set $\Bbb
P'_\gamma$ taken as $\matop{\lim}{\rightarrow}{}\{\Bbb
P'_{\gamma'}\mid\gamma'<\gamma\}$, and $\mps{H_\gamma }{\Bbb
P_\gamma}{ \Bbb P'_{\gamma}}$ taken as $H_{\gamma'} $ on $\Bbb
P_{\gamma'}$ for all  $\gamma'<\gamma$. In view of  $(3)$ and
local finiteness of $\{\Bbb E_{\gamma}\}$, $\Bbb
P_\gamma=\matop{\lim}{\rightarrow}{}\{\Bbb
P_{\gamma'}\mid\gamma<\gamma\}$ is a closed neighborhood of $\x$
in $\Bbb Q_{\gamma}$ and, therefore, $\mps{H_\gamma }{\Bbb
P_\gamma}{ \Bbb P'_{\gamma}}$ is an continuous equigeny. Since
$\Bbb X\hookrightarrow\Bbb P'_\gamma$ is a closed embedding, $\Bbb
P_{\gamma'}\subset\Bbb P_\gamma$ and
$H_{\gamma'}=H_{\gamma}\restriction_{\Bbb P_{\gamma'}}$ for all
$\gamma'<\gamma$, the proof of Theorem \ref{three+++} is
completed.
\medskip

\section{ Some remarks on Murayama's paper  "On  G-ANRs and their G-homotopy
 types"}
 The paper of Murayama \cite{Mu} contains the basis of the theory of equivariant absolute
extensors and retracts for spaces  with action of compact Lie
group. As an application of the developed methods he has proved a
series of important results (which remain of interest up to date):
\begin{enumerate}\setcounter{enumi}{0}
  \item [$(a)$] Each convex $G$-subset $C$ of a locally convex topological linear space $L$ is an equivariant
absolute extensor; and
 \item [$(b)$] Each equivariant absolute neighborhood extensor has
the $G$-homotopy type  of a $G$-$CW$-complex.
 \end{enumerate}
Since the reasoning behind the Murayama's proof was not based on
assumption of $L$ being a $G$-space, we formulate Theorem (a) more
generally than in \cite[Theorem 5.3]{Mu} permitting $L$ to be
merely a locally convex topological linear space without any
action away from $C$.

 As a several of original papers, also this one contains some
mistakes (which, however, do not affect the main results):
 \begin{enumerate}\setcounter{enumi}{0}
 \item [$(c)$] We have pointed out that \cite[Proposition 8.7(2)]{Mu} was proved only
for an Abelian Lie group.
 \item [$(d)$] The proof of the necessity
of \cite[Proposition 8.5]{Mu} is also performed for an Abelian
Lie group. In the general case one should slightly improve it:
take sufficiently small neighborhood $U$ invariant with respect to
both left regular and right regular actions of group $H$ on $G$
(the existence of which is established straightforwardly); then
take a retraction $r$ equivariant with respect to both right
regular action of $H$ on $U$ and left regular action of $H$ on
$U$, and further proceed according to the Murayama's original
proof.
 \item [$(e)$] The main (but easily avoided) deficiency is \cite[Proposition  5.1]{Mu},
asserting  that the Banach space  $B(X)$ of all bounded functions
on the $G$-space $X$ is a $G$-space  (this is true only for
compact  $X$).
 \end{enumerate}
 However Murayama is not the first who made the latter
error -- see, for example, Jaworowskii \cite{Jw}. In fact, proving
an equivariant analog of Wojdyslawskii's theorem, both Murayma and
Jaworowskii have used only the existence of an equivariant
embedding of $X$ into the convex hull $C$ of $X$ which naturally
lies in $B(X)$ (see \cite[Theorem 6.2]{Mu} and \cite[Proposition
4.1]{Jw}). Though the action of $G$ on $B(X)$ is discontinuous,
it can however be easily checked that the restriction of this
action onto the convex hull $C$ is continuous. 

In view of Theorem
(a) it completely rehabilitates \cite[Theorem 6.2]{Mu} and, in a
literal sense, the deficiency (c) does not affect the validity of
all facts proved later on. Without any doubt, Murayama can be
considered to have proved the equivariant Wojdyslawskii's theorem.
In view of this observation, other places in \cite{Mu} --
Proposition 8.1, Theorem 6.4, Proposition 10.1 (on
$G$-domination), Corollary 10.2, Theorem 11.1 and so on, can be
considered as completely proved.

 Recently some authors have raised some doubt
 concerning the substantiation of the
results of the Murayama's paper (see \cite{An},\cite{AE} and
certain
other papers by
the same authors). As to the deficiency (e), we have
already
explained that it does not affect the main results.

 There has also been
 some concern (see \cite{An}) regarding the
continuity both of the action on the $G$-nerve $K(S)$ and of the
map $P$ into this $G$-nerve in \cite[Proposition 2.4]{Mu}.
Therefore one can raise some doubt  about the validity of the main
results -- Theorem (a) and Theorem (b).

To dispel with the first doubt one should consider the theory of
simplicial sets where proofs of similar facts are straightforward exercises.
The second doubt (the continuity of $P$) is also groundless: the
proof of \cite[Proposition 2.4]{Mu} was executed flawlessly and it is
based on results of Segal on classifying spaces \cite{Sg}. Only at
the end of this proof there is a misprint -- one should rearrange
the order of maps in the composite. Incidentally, in \cite{AE}
this proof added by Segal's results (to which Murayama only
made a reference) was reproduced word for word. Therefore, in our
opinion, Theorems (a) and (b) without any doubt belong to
Murayama.

\section*{Acknowledgements}
The first author was supported by a grant from
 the Ministry of Education of the Republic of Belarus.
 The second author were supported by a grant from
 the Slovenian Research Agency.
 We thank the referee for comments and suggestions.


\begin{thebibliography}{99}

 \bibitem{Ab}
 H. Abels,
 {\it A universal proper $G$-space,}
 Math. Z.  {\bf 159} (1978), 143--158.

\bibitem{AU}
S. M. Ageev,
{\it An equivariant Dugundji theorem,}
Uspehi Matem Nauk\
{\bf 45}: 5 (1990), 179--180;
English transl. in
Russian Math. Surveys {\bf  45}:5 (1990), 219--220.

\bibitem{AgIzv}
S. M. Ageev,
{\it Classification of G-spaces,}
Izv. Ross. Akad. Nauk Ser. Mat. {\bf 56}:6 (1992), 1345--1357;
English transl.
in Russian Acad. Sci. Izv. Math. {\bf 41}:3 (1993), 581--591.

\bibitem{Ag2}
S. M. Ageev,
{\it On equivariant homotopy type,}
Latv. Univ. Zinat. Raksti {\bf 576} (1992), 37--44.

\bibitem{Ag1}
S. M. Ageev,
{\it Characterization of $G$-spaces}, Doctoral Thesis, Moscow
State University, Moscow, 1996. (in Russian)

\bibitem{Ag}
 S. M. Ageev,
 {\it  Isovariant absolute extensors,} 
 Mat.  Sbornik, submitted.

\bibitem{AgRep}
 S. M. Ageev and D. Repov\v{s},
 {\it  On extension of group actions,}
Mat.  Sbornik \ {\bf 201}:2 (2010), 3--29;
English transl. in  Sbornik Math. {\bf 201}:2 (2010), 159--182.

\bibitem{An}
S. A. Antonyan,
{\it Equivariant embeddings into $G$-AR's,}
Glasnik
Mat. {\bf 22(42)} (1987), 503--533.

\bibitem{AE}
S. A. Antonyan and E. Elfving,
 {\it The equivariant homotopy type of $G$-ANRs for compact group actions,}
 Manuscr. Math.  {\bf 124} (2007), 275--297.

\bibitem{Br}
 G. E. Bredon,
 {\it Introduction to Compact Transformation Groups},
 Pure and Appl. Math. {\bf 46}, Academic Press, New York, 1972.

 \bibitem{Dc}
T. tom Dieck,
{\sl Transformation Groups and Representation Theory},
Springer-Verlag, Berlin, 1979.

\bibitem{DT}
T. tom Dieck,
 {\it Transformation Groups,}
 Walter de Gruyter, Berlin, 1987.

\bibitem{Du}
J. Dugundji,
{\it Absolute neighborhood retracts and local connectedness in arbitrary metric spaces,}
Comp. Math. {\bf 13} (1958), 229--246.

\bibitem{Eng}
R. Engelking,
{\sl General Topology},
PWN, Warsaw, 1977.

\bibitem{Jw}
J. Jaworowski,
{\it Extensions of $G$-maps and Euclidean $G$-retracts,}
Math. Zeitschr. {\bf 146} (1976), 143--148.

 \bibitem{Matu} 
 T. Matumoto,
 {\it Equivariant $K$-theory and Fredholm operator,} 
 J. Fac. Sci. Univ. Tokyo Sect. IA Math.  {\bf 18} (1971), 109--125.

\bibitem{Mu}
M. Murayama,
 {\it On  $G$-ANRs and their $G$-homotopy types,}
 Osaka J. Math.  {\bf 20}:3 (1983), 479--512.

 \bibitem{Pal}
R. Palais,
 {\it   The classification of $G$-spaces,}
Memoirs Amer. Math. Soc.  {\bf  36}, Providence, R.I., 1960.

 \bibitem{Ptr}
L. S. Pontryagin,
{\sl Topological Groups},
Gordon and Breach, New York, 1966.

\bibitem{Sg}
 G. Segal,
 {\it Classifying spaces and spectral sequences,}
Publ. Math. I. H. E. S. {\bf 34} (1968), 105--112.

\bibitem{Sp}
 E. H. Spanier,
 {\sl Algebraic Topology},
 Mc Graw-Hill, New York, 1966.

\end{thebibliography}
\end{document}